\documentclass[12pt]{article}
\usepackage{amssymb}
\usepackage{amscd}
\usepackage{amsfonts}
\usepackage{times}
\usepackage{mathptmx}
\usepackage{amsmath}
\usepackage[usenames]{color}
\usepackage{mathrsfs}
\usepackage{amsfonts}
\usepackage{amssymb,amsmath}
\usepackage{CJK}
\usepackage{cite}
\usepackage{cases}
\usepackage{amsthm}
\usepackage{extarrows}
\usepackage{indentfirst,latexsym,bm}
\usepackage{amsthm}
\def\cl{\centerline}
\def\la{\lambda}
\def\a{\alpha}

\def\pa{\partial}
\def\vs{\vspace*}

\def\LL{\mathscr{L}}

\def\Z{\mathbb{Z}}

\def\C{\mathbb{C}}

\def\vs{\vspace*}

\numberwithin{equation}{section}
\newtheorem{theo}{Theorem}[section]
\newtheorem{defi}[theo]{Definition}

\newtheorem{lemm}[theo]{Lemma}
\newtheorem{prop}[theo]{Proposition}
\newtheorem{rem}[theo]{Remark}

\begin{document}
\begin{center}
{\bf\large Conformal modules and their extensions of a Lie conformal algebra related to a 2-dimensional Novikov algebra}
\footnote {$^{\,*}$Corresponding author: lmyuan@hit.edu.cn (Lamei Yuan )}
\end{center}

\cl{Lamei Yuan$^{\,*}$,Yanjie Wang}

\cl{\small School of Mathematics, Harbin Institute of Technology, Harbin
150001, China}

\cl{\small\footnotesize E-mails:lmyuan@hit.edu.cn 18s012006@stu.hit.edu.cn}
\vs{8pt}

{\small
\parskip .005 truein
\baselineskip 3pt \lineskip 3pt
\noindent{{\bf Abstract:} Let $\mathcal{R}$ be a free Lie conformal algebra of rank $2$ with $\C[\pa]$-basis $\{L,I\}$ and relations
\begin{eqnarray*}
\left[L_{\lambda} L\right]=(\partial+2 \lambda) (L+I),\ \left[L_{\lambda} I\right]=(\partial+\lambda) I, \ \left[I_{\lambda} L\right]=\lambda I,\ \left[I_{\lambda} I\right]=0.
\end{eqnarray*}
In this paper, we first classify all finite nontrivial irreducible conformal modules over $\mathcal{R}$. Then we determine extensions between two finite irreducible conformal $\mathcal{R}$-modules.
 \vs{5pt}

\noindent{\bf Key words:} Conformal modules, Extensions

\vs{5pt}

\noindent{\bf MR(2000) Subject Classification:}~ 17B65, 17B68

\parskip .001 truein\baselineskip 6pt \lineskip 6pt

\section{Introduction}
\vs{8pt}

Conformal module is a basic tool for the construction of free field realization of infinite-dimensional Lie (super)algebras in conformal field theory.
All finite irreducible conformal modules over the Virasoro, the Neveu-Schwarz and the current conformal algebras were studied and constructed in \cite{CK}. The same problem for a class of Lie conformal algebras of rank two, the Schr\"odinger-Virasoro conformal algebra and a class of Lie conformal algebras of Block type was solved in \cite{SXY,WY,YL}.

In general, conformal modules are not completely reducible. In order to understand the representation theory of Lie conformal algebras,
one is led to study the extension problem. Cheng, Kac and Wakimoto classified extensions between finite conformal modules over the Virasoro conformal algebra, the current conformal algebra and their semidirect sum in \cite{CKW1}. Later they further solved the same problem for the Neveu-Schwarz conformal algebra in \cite{CKW2}. Based on the
theory of conformal modules given in \cite{CK} and the techniques developed in \cite{CKW1}, Ngau Lam
classified extensions between finite conformal modules over the supercurrent conformal algebra in \cite{L}, and we did the same for the
Heisenberg-Virasoro and the (extended) Schr\"odinger-Virasoro conformal algebras in \cite{LY,YL}.

A Novikov
algebra $\mathcal{A}$ is a vector space over $\mathbb{C}$ with a bilinear product $(a, b) \rightarrow a b$ satisfying
\begin{eqnarray*}
(a b) c-a(b c)=(b a) c-b(a c),\
(a b) c=(a c) b,\ \forall \ a, b, c \in \mathcal{A}.
\end{eqnarray*}
Consider the space
$L(\mathcal{A})=\mathcal{A} \otimes \mathbb{C}\left[t, t^{-1}\right]$, whose elements are written as $a \otimes t^{m}$, where $a\in\mathcal{A}$ and $m\in\mathbb{Z}$. Denote $a[m] =a \otimes t^{m+1}$. Define
a bilinear product $[\cdot,\cdot]: L(\mathcal{A}) \times L(\mathcal{A}) \rightarrow L(\mathcal{A})$ by
\begin{eqnarray}\label{1}
[a[m], b[n]]=(m+1)(a b)[m+n]-(n+1)(b a)[m+n],\ \forall \ a,b\in\mathcal{A},\ m,n\in\mathbb{Z}.
\end{eqnarray}
It was shown in \cite{BN} that $L(\mathcal{A})$ is a Lie algebra if and only if $\mathcal{A}$ is a Novikov algebra. Then $L(\mathcal{A})$ is called the Lie algebra associated with $\mathcal{A}$.

Let $\mathcal{N}$ be the two-dimensional Novikov algebra with basis $\{e_1,e_2\}$ and relations
\begin{eqnarray}
e_1e_1=0, \ \ e_1e_2=e_1,\ \ e_2e_1=0, \ \ e_2e_2=e_1+e_2.
\end{eqnarray}
Set $L_{m}=e_{2}[m],$ $I_{m}=e_{1}[m]$.
By \eqref{1}, the associated Lie algebra $L(\mathcal{N})$ is an infinite-dimensional Lie algebra with basis $\{L_m, I_m|m\in\Z\}$ and commutation relations
\begin{eqnarray}\label{2}
[L_{m}, L_{n}]=(m-n) L_{m+n}+(m-n) I_{m+n},\ \left[L_{m}, I_{n}\right]=-(n+1) I_{m+n},\ \left[I_{m}, I_{n}\right]=0.
\end{eqnarray}
All derivations and central extensions of this Lie algebra were
determined in \cite{PB}. Considering the following $L(\mathcal{N})$-valued formal distributions
$$L(z)=\sum_{n \in \mathbb{Z}} L_{n} z^{-n-2},\quad I(z)=\sum_{n \in \mathbb{Z}} I_{n} z^{-n-2},$$
the commutation relations in \eqref{2} are equivalent to the following ones
\begin{eqnarray}
[L(z), L(w)]&=&\partial_{w} (L(w)+I(w)) \delta(z, w)+2 (L(w)+I(w)) \partial_{w}\delta(z, w),\label{lca1}\\
{[L(z), I(w)]}&=&\partial_{w} I(w) \delta(z, w)+I(w) \partial_{w} \delta(z, w),\label{lca2-1}\\
{[I(z), L(w)]}&=&I(w) \partial_{w} \delta(z, w),\label{lca2-2}\\
{[I(z), I(w)]}&=&0,\label{lca3}
\end{eqnarray}
where $\delta(z, w)=\sum_{n \in \mathbb{Z}} z^{-n-1} w^{n}$. Therefore, $\mathcal{F}=\{L(z),I(z)\}$ is a local family of $L(\mathcal{N})$-valued formal distributions and $L(\mathcal{N})$ is a formal distribution Lie algebra. Furthermore,
in term of the $\lambda$-bracket, relations \eqref{lca1}--\eqref{lca3} translate as follows:
\begin{eqnarray}\label{lam}
\left[L_{\lambda} L\right]=(\partial+2 \lambda) (L+I),\ \left[L_{\lambda} I\right]=(\partial+\lambda) I, \ \left[I_{\lambda} L\right]=\lambda I,\ \left[I_{\lambda} I\right]=0.
\end{eqnarray}
Denote $\mathcal{R}=\mathbb{C}[\partial] L \bigoplus \mathbb{C}[\partial] I$, which is a free $\mathbb{C}[\partial]$-module of rank 2.
One can check that $\mathcal{R}$ equipped
with the $\lambda$-brackets in \eqref{lam} forms a Lie conformal algebra.

In this paper, we aim to classify all finite irreducible conformal $\mathcal{R}$-modules and extensions between them.
In Section 2, we recall the notions of  conformal modules and their extensions, and some known results about Virasoro conformal modules. Section 3 is devoted to a classification of all finite irreducible conformal $\mathcal{R}$-modules by using the equivalent language of Lie conformal algebras and extended annihilation algebras, and some techniques developed in \cite{CK,WY}. We will show that any finite irreducible conformal module over $\mathcal{R}$ is free of rank one and of the form $V(\alpha,\Delta)=\C[\pa]v_{\Delta}$ with $\Delta \neq 0$ and the action of $\mathcal{R}$ given by
$L_\lambda
v_\Delta=(\partial+\alpha+\Delta \lambda)v_\Delta, \ I_\lambda v_\Delta=0.$ In Section 4, we study
 extensions of finite conformal $\mathcal{R}$-modules of the following types:
\begin{eqnarray}
&&0\longrightarrow \C{c_\gamma}\longrightarrow E \longrightarrow V(\alpha,\Delta) \longrightarrow 0,\label{stype1}\\
&&0\longrightarrow V(\alpha,\Delta)\longrightarrow E \longrightarrow \C c_\gamma \longrightarrow 0,\label{stype2}\\
&&0\longrightarrow V(\bar\alpha,\bar\Delta)\longrightarrow E \longrightarrow V(\alpha,\Delta) \longrightarrow 0, \label{stype3}
\end{eqnarray}
where $\C{c_\gamma}$ denotes the 1-dimensional $\mathcal{R}$-module with $L_\lambda c_\gamma=I_\lambda c_\gamma=0$ and $\partial c_\gamma=\gamma c_\gamma$. We will give explicit cocycles corresponding to trivial or nontrivial
extensions.

Throughout the paper, all vector spaces, tensor products and algebras are assumed to be over the field of complex numbers $\C$. In addition to the standard notation $\Z$, we use $\Z^+$ (resp. $\C^*$)  to denote the set of nonnegative integers (resp. nonzero complex numbers).

\section{Preliminaries}
\vs{8pt}
In this section, we first begin by recalling the definitions of Lie conformal algebras, conformal modules and their extensions, and some results that we need in this paper. For more details, the reader is referred to \cite{CK,CKW1,WY}.
\subsection{Lie conformal algebra}

\begin{defi}\rm
A Lie conformal algebra $R$ is a $\C[\partial]$-module endowed with a $\C$-bilinear map
$$ R\otimes R\rightarrow \C[\lambda]\otimes R,\ \  a\otimes b \mapsto [a_\lambda b],$$
called the $\la$-bracket, and
satisfying the following axioms ($a, b, c\in R$):
\begin{eqnarray}
[\partial a_\lambda b]&=&-\lambda[a_\lambda b],\ \ [ a_\lambda \partial b]=(\partial+\lambda)[a_\lambda b] \ \ \mbox{(conformal\  sesquilinearity)},\label{Lc1}\\
{[a_\lambda b]} &=& -[b_{-\lambda-\partial}a] \ \ \mbox{(skew-symmetry)},\label{Lc2}\\
{[a_\lambda[b_\mu c]]}&=&[[a_\lambda b]_{\lambda+\mu
}c]+[b_\mu[a_\lambda c]]\ \ \mbox{(Jacobi \ identity)}\label{Lc3}.
\end{eqnarray}
\end{defi}

Let $R$ be a Lie conformal algebra.
For each $j\in\Z^+$, we can define the {\it $j$-product} $a_{(j)}b$ of any two elements $a,b\in R$ by the following generating series:
\begin{equation}\label{jj1}
[a_{\lambda}b]=\sum_{j\in\Z^{+}}(a_{(j)}b)\frac{\lambda^j}{j!}.
\end{equation}
Then the following axioms of $j$-products hold:
\begin{equation}\label{lcaj}
\aligned
&a_{(n)}b=0,\ {\rm for}\ n\gg0,\\
&(\partial a)_{(n)}b=-na_{(n-1)}b,\\
&a_{(n)}b=\sum_{j\in\Z^{+}}(-1)^{n+j+1}\frac{\partial^j}{j!} b_{(n+j)}a,\\
&[a_{(m)},b_{(n)}]=\sum_{j=0}^m \begin{pmatrix}
m\\j
\end{pmatrix}(a_{(j)}b)_{(m+n-j)}.
\endaligned
\end{equation}
Actually, one can also define Lie conformal algebras using the language of $j$-products (c.f. \cite{Kac1}).

Consider the space $\widetilde{R}=R\bigotimes\C[t,t^{-1}]$ with $\widetilde{\partial}=\partial\otimes{\rm id}+{\rm id}\otimes\partial_t$, where $\rm id$ appearing on the left
(resp. right) of $\bigotimes$ is the identity operator acting on $R$ (resp. $\C[t,t^{-1}]$). This space is called the {\it affinization}
of $R$. Its generating elements can be written as $a\otimes t^m$, where $a\in R$ and $m\in\Z$. For clarity, we will use the notation $\widetilde{R}=R[t,t^{-1}]$, $a t^m$ for its elements and $\widetilde{\partial}=\partial+\partial_t$. By \eqref{lcaj}, we obtain a well defined commutation relation on $\widetilde{R}$:
\begin{equation}\label{ppp2}
[at^m, bt^n]=\sum_{j\in\Z^{+}}
\begin{pmatrix}
m\\j
\end{pmatrix}
(a_{(j)}b)t^{m+n-j}, \ \forall\ at^m,\, bt^n\in \widetilde{R},
\end{equation}
which gives $\widetilde{R}$ a structure of algebra, denoted by $( \widetilde{R}, [\cdot,\cdot])$.
It can be verified that the subspace $\widetilde{\partial}\widetilde{R}$ spanned by elements of the form $\{(\pa a) t^n+n a t^{n-1}| n\in\Z\}$ is a  two-sided ideal of the algebra $( \widetilde{R}, [\cdot,\cdot])$. Set
\begin{equation}
Lie(R)={\widetilde{R}}/{\widetilde{\partial}\widetilde{R}}.
\end{equation}
Let $a_{(m)}$ denote the image of $at^m$ in $Lie(R)$. Then $(\pa a)_{(n)}=-n a_{(n-1)}$. Define a bracket on $Lie(R)$ by
\begin{eqnarray}\label{lie}
[a_{(m)},b_{(n)}]=\sum_{j=0}^m \begin{pmatrix}
m\\j
\end{pmatrix}(a_{(j)}b)_{(m+n-j)},
\end{eqnarray}
for $a,b\in R$, $m, n\in\Z$. One can check that $(Lie(R),[\cdot,\cdot])$ is a Lie algebra with respect to \eqref{lie}. Note that $Lie(R)$ admits a derivation $\pa$ defined by $\pa(a_{(n)})=-n a_{(n-1)}$, for $a\in R$ and $n\in\Z$.
The Lie subalgebra $\textit{Lie}(R)^+={\rm span}_{\C}\{a_{(m)}\,|\,a\in R, m\in\Z^+\}$ is called the {\it annihilation algebra} of $R$.
The {\it extended annihilation algebra} $\textit{Lie}(R)^e$ is defined as the
semidirect product of the $1$-dimensional Lie subalgebra $\C \pa$ and $\textit{Lie}(R)^+$ with the action
$[\pa,a_{(n)}]=-na_{(n-1)}$.

\subsection{Conformal modules}

\begin{defi}\label{def1} \rm A {\it conformal module} $V$ over a Lie conformal algebra $R$
is a $\mathbb{C}[\partial]$-module endowed with a $\C$-linear map
$R\rightarrow {{\rm End_{\mathbb{C}}}(V) \bigotimes_{\mathbb{C}}\mathbb{C}[\lambda]}$, $a\mapsto a_\lambda $, satisfying the following conditions for all $a,b\in R$:
\begin{eqnarray*}
[a_\lambda,b_\mu]=a_\lambda b_\mu -b_\mu a_\lambda =[a_\lambda b]_{\lambda+\mu},\
(\partial a)_\lambda =[\partial,a_\lambda]=-\lambda a_\lambda.
\end{eqnarray*}
A conformal module $V$ over a Lie conformal algebra $R$ is called {\it finite}
if $V$ is finitely generated over
$\mathbb{C}[\partial]$. It is called {\it irreducible} if there is no nontrivial invariant subspace.
\end{defi}

Let $V$ be a conformal module over a Lie conformal algebra $R$.
An element $v$ in $V$ is called an {\it invariant} if $R_\la v=0$.
Denote by $V^0$ the subspace of invariants of $V$.
It is easy to see that $V^0$ is a conformal submodule of $V$.
If $V^0=V$, then $V$ is called a {\it trivial} $R$-module.
A trivial module just admits the structure of a $\C[\partial]$-module. The vector space $\C$ is viewed as a trivial module with trivial actions of both $\pa$ and $R$.
For any fixed complex constant $\alpha$, there
is a natural trivial $R$-module $\C c_\alpha$, such that $\C c_\alpha=\C$ and $\pa v=\alpha v,\ R_\la v=0$ for $v\in\C c_\alpha$.
The modules
$\C c_\alpha$ (with $\alpha\in\C^*$) exhaust all trivial irreducible $R$-modules.

The following result is due to \cite[Lemma 2.2]{Kac2}.\vs{-8pt}

\begin{lemm}
Let $R$ be a Lie conformal algebra and $V$ an $R$-module.\\
1) If $\partial v= a v$ for some $a\in \C$ and $v\in V$, then $R_\la v=0$.\\
2) If $V$ is a finite conformal module without any nonzero invariants, then $V$ is a free $\C[\partial]$-module.
\end{lemm}

An element $v\in V$ is called a {\it torsion element} if there exists a nonzero polynomial $p(\partial)\in\C[\partial]$ such that $p(\partial)v=0$.
For any $\C[\partial]$-module $V$, there exists a nonzero torsion element if and only if there exists a nonzero $v\in V$
such that $\partial v=a v$ for some $a\in\C$.
A finitely generated $\C[\partial]$-module $V$ is free if and only if $0$ is the only torsion element of $V$.

We deduce from the above discussion the following important result.

\begin{lemm}\label{I}
Let $R$ be a Lie conformal algebra.
Then every finite nontrivial irreducible $R$-module has no nonzero torsion element and
is a free $\C[\partial]$-module.
\end{lemm}


Assume that $V$ is a conformal module over $R$.
We can also define {\it $j$-actions} of $R$ on $V$ using the following generating series
\begin{equation}
a_{\lambda}v=\sum_{j\in\Z^{+}}(a_{(j)}v)\frac{\lambda^j}{j!}.
\end{equation}
The $j$-actions satisfy relations similar to those in (\ref{lcaj}).

It is immediate to see that a conformal module $V$ over a Lie conformal algebra $R$ is the same as a module over the extended annihilation algebra $\textit{Lie}(R)^e$ satisfying the following local nilpotent condition
\begin{equation}\label{conformal}
a_{(n)}v=0,\ \ for\ \ v\in V,\ a\in R, \ n\gg 0.
\end{equation}
A $\textit{Lie}(R)^e$-module satisfying this condition is called {\it conformal}.

It was shown in
\cite{CK} that all free nontrivial Virasoro conformal modules of rank one
over $\mathbb{C}[\partial]$ are the following ones $(\Delta,
\alpha\in \mathbb{C})$:
\begin{eqnarray}
M(\alpha,\Delta)=\mathbb{C}[\partial]v_\Delta,\ \ L_\lambda
v_\Delta=(\partial+\alpha+\Delta \lambda)v_\Delta.
\end{eqnarray}
The module $M(\alpha,\Delta)$ is irreducible if and only if
$\Delta\neq 0$. The module $M(\alpha,0)$ contains a unique
nontrivial submodule $(\partial +\alpha)M(\alpha,0)$ isomorphic to
$M(\alpha,1).$  Moreover, the modules $M(\alpha,\Delta)$ with
$\Delta\neq 0$ exhaust all finite non-1-dimensional irreducible
Virasoro conformal modules. Therefore $M(\alpha,\Delta)$ with $\Delta\neq 0$, together with the one-dimensional modules $\C c_\beta$ ($\beta\in\C^*$), form a complete list of finite irreducible conformal modules over the Virasoro conformal algebra.

\subsection{Extensions}
Let $V$ and $W$ be two modules over a Lie conformal algebra (or a Lie algebra) $R$. An {\it extension} of $W$ by $V$ is an exact sequence of $R$-modules of the form
\begin{eqnarray}\label{Em}
0\longrightarrow V\xlongrightarrow{i} E \xlongrightarrow{p} W \longrightarrow 0.
\end{eqnarray}
Two extensions $0\longrightarrow V\xlongrightarrow{i} E \xlongrightarrow{p} W \longrightarrow 0$ and $0\longrightarrow V\xlongrightarrow{i'} E' \xlongrightarrow{p'} W \longrightarrow 0$ are said to be {\it equivalent} if there exists a commutative diagram of the form
\begin{equation*}
\begin{CD}
0@>>> V @>i>{\rm }>  E @>p>> W@>>> 0\\
@. @V{1_V}VV @V\psi VV @V{1_W}VV\\
0@>>> V @>i'>{\rm }> E' @>p'>>
W @>>> 0,
\end{CD}
\end{equation*}
where $1_V: V\rightarrow V$ and $1_W: W\rightarrow W$ are the respective identity maps and $\psi: E\rightarrow E'$ is a homomorphism of modules.

The direct sum of modules $V\oplus W$ obviously gives rise to an extension. Extensions equivalent to it are called {\it trivial extensions}. In general, an extension can be thought of as the direct sum of vector spaces $E=V\oplus W$, where $V$ is a submodule of $E$, while for $w$ in $W$ we have:
\begin{equation*}
a\cdot w=aw+\phi_a(w),\ \ a\in R,
\end{equation*}
where $\phi_a:W\rightarrow V$ is a linear map satisfying the cocycle condition: $$\phi_{[a,b]}(w)=\phi_a(bw)+a\phi_b(w)-\phi_b(aw)-b\phi_a(w),\ \, b\in R.$$ The set of these cocycles form a vector space over $\C$. Cocycles equivalent to the trivial extension are called {\it coboundaries}. They form a subspace and the quotient space by it is denoted by $\textrm{Ext}(W, V).$

In \cite{CKW1}, extensions of finite irreducible Virasoro conformal modules of the following three types have been classified ($\Delta, \bar\Delta\in\C^*$):
\begin{eqnarray}
&&0\longrightarrow \C{c_\gamma}\longrightarrow E \longrightarrow M(\alpha,\Delta) \longrightarrow 0,\label{type1}\\
&&0\longrightarrow M(\alpha,\Delta)\longrightarrow E \longrightarrow \C c_\gamma \longrightarrow 0,\label{type2}\\
&&0\longrightarrow M(\bar\alpha,\bar\Delta)\longrightarrow E \longrightarrow M(\alpha,\Delta) \longrightarrow 0. \label{type3}
\end{eqnarray}
We list the corresponding results in the following three theorems.

\begin{theo}\label{th2}
Nontrivial extensions of the form \eqref{type1}\ exist if and only if  $\alpha+\gamma=0$ and $\Delta=1$ or $2$. In these cases, they are given (up to equivalence) by $$L_\lambda v_\Delta=(\partial+\alpha+\Delta\lambda)v_\Delta+f(\lambda)c_\gamma,$$ where
\begin{itemize}
\item[{\rm (i)}] $f(\lambda)=c_2\lambda^2$, for $\Delta=1$ and $c_2\neq0$.
\item [{\rm (ii)}] $f(\lambda)=c_3\lambda^3$, for $\Delta=2$ and $c_3\neq0$.
\end{itemize}
Furthermore, all trivial cocycles are given by scalar multiples of the polynomial $f(\lambda)=\alpha+\gamma+\Delta\lambda$.
\end{theo}

\begin{theo}\label{th3}
Nontrivial extensions of Virasoro conformal modules of the form \eqref{type2} exist if and only if $\alpha+\gamma=0$ and $\Delta=1$. These extensions are given (up to equivalence) by

\begin{eqnarray*}
L_\lambda c_\gamma=f(\partial,\lambda)v_\Delta,\ \
\partial c_\gamma=\gamma c_\gamma+h(\partial)v_\Delta,
\end{eqnarray*}
where $f(\partial,\lambda)=h(\partial)=a_0\in\C^*$.
\end{theo}

\begin{theo}\label{th4} Nontrivial extensions of Virasoro conformal modules of the form \eqref{type3} exist only if $\alpha=\bar\alpha$ and $\Delta-\bar\Delta=0,2,3,4,5,6.$ In these cases, they are given (up to equivalence) by $$L_\lambda v_\Delta=(\partial+\alpha+\Delta\lambda)v_\Delta+f(\partial,\lambda)v_{\bar\Delta },$$
where the values of $\Delta$ and $\bar\Delta$ along with the corresponding polynomials $f(\pa,\la)$ whose nonzero scalar multiples give rise to nontrivial extensions are listed as follows ( $\bar\partial=\partial+\alpha$):
\begin{itemize}
\item[{\rm (i)}] $\Delta=\bar\Delta, f(\pa,\la)=c_0+c_1\la, (c_0,c_1)\neq(0,0).$
\item [{\rm (ii)}]$\Delta-\bar\Delta=2, f(\pa,\la)=\la^2(2\bar\pa+\la)$.
\item [{\rm (iii)}]$\Delta-\bar\Delta=3, f(\pa,\la)=\bar\pa\la^2(\bar\pa+\la)$.
\item [{\rm (iv)}]$\Delta-\bar\Delta=4, f(\pa,\la)=\la^2(4\bar\pa^3+6\bar\pa^2\la-\bar\pa\la^2+\bar\Delta\la^3)$.
\item [{\rm (v)}]$(\Delta,\bar\Delta)=(1,-4), f(\pa,\la)=\bar\pa^4\la^2-10\bar\pa^2\la^4-17\bar\pa\la^5-8\la^6$.
\item [{\rm (vi)}]$(\Delta,\bar\Delta)=(\frac72\pm\frac{\sqrt{19}}2, -\frac52\pm\frac{\sqrt{19}}2),  f(\pa,\la)=\bar\pa^4\la^3-(2\bar\Delta+3)\bar\pa^3\la^4-3\bar\Delta\bar\pa^2\la^5-(3\bar\Delta+1)\bar\pa\la^6-(\bar\Delta+\frac9{28})\la^7$.
\end{itemize}
\end{theo}

\begin{rem} {\rm We rewrite \cite[Theorem 2.4]{CKW2} (see also \cite[Theorem 3.2]{CKW1}) as Theorem \ref{th4} by removing the cases $\Delta=0$ or $\bar\Delta=0$. The reason is that both $M(\alpha,\Delta)$ and $M(\bar\alpha,\bar\Delta)$ in \eqref{type3} are irreducible Virasoro conformal modules.}
\end{rem}

\section{Classification of finite nontrivial irreducible $\mathcal{R}$-modules}
This section is devoted to classification of all finite nontrivial irreducible conformal modules over the Lie conformal algebra $\mathcal{R}$.

\begin{prop}\label{key1}
\begin{itemize}
\item[{\rm (1)}]
All nontrivial conformal $\mathcal{R}$-modules of rank 1
are the following ($\Delta,\a\in\C$):
\begin{eqnarray}\label{key-key}
V(\alpha,\Delta)=\mathbb{C}[\partial]v,\quad L_\lambda
v=(\partial+\alpha+\Delta \lambda)v, \quad
I_\lambda v=0.
\end{eqnarray}
\item[{\rm (2)}]
The module
$V(\alpha,\Delta)$ is irreducible if and only if $\Delta\neq 0$. Moreover, the module $V(\alpha,0)$ contains a unique
nontrivial submodule $(\partial +\alpha)V(\alpha,0)$ isomorphic to
$V(\alpha,1).$
\end{itemize}
\end{prop}
\begin{proof} (1) Suppose that
\begin{eqnarray}\label{m1}
L_{\lambda} v=f(\partial, \lambda) v,\quad I_{\lambda} v=g(\partial, \lambda) v,\end{eqnarray}
where $f(\partial, \lambda)$ and $g(\partial, \lambda)$ are polynomials in $\mathbb{C}[\partial,\lambda]$. By $[I_\lambda I]=0$, we have $I_{\lambda}\left(I_{\mu} v\right)=I_{\mu}\left(I_{\lambda} v\right)$. Thus $$g(\partial+\lambda,\mu)g(\partial,\lambda)=g(\partial+\mu,\lambda)g(\partial,\mu).$$
This implies $\operatorname{deg}_{\lambda} g(\partial, \lambda)+\operatorname{deg}_{\partial} g(\partial, \lambda)=\operatorname{deg}_{\lambda} g(\partial, \lambda),$ where we use the notation $\operatorname{deg}_{\lambda} g(\partial, \lambda)$
to denote the highest degree of $\lambda$ in $g(\partial, \lambda) .$ Hence ${\rm deg}_{\partial}g(\partial,\lambda)=0,$ and we can suppose that $g(\partial, \lambda)$
$=g(\lambda)$ for some $g(\lambda) \in \mathbb{C}[\lambda].$ Then applying both sides of $\left[L_{\lambda} L\right]=(\partial+2\lambda)(L+I)$ to $v$ gives
\begin{eqnarray}
f(\partial+\lambda,\mu)f(\partial,\lambda)-f(\partial+\mu,\lambda)f(\partial,\mu)
=(\lambda-\mu)\big(f(\partial,\lambda+\mu)+g(\lambda+\mu)\big).\label{m1-1}
\end{eqnarray}
If $f(\partial,\lambda)=0$, then $g(\partial,\lambda)=g(\lambda)=0$. In this case, $V(\alpha,\Delta)$ is a trivial $\mathcal{R}$-module. A contradiction.
Thus $f(\partial,\lambda)\neq 0$. Write $f(\partial,\lambda)=\sum_{i=0}^{n} a_{i}(\lambda) \partial^{i}$ with $a_{n}(\lambda)\neq 0.$ Then, assuming $n>1,$ if we equate terms of degree $2n-1$ in $\partial$ in \eqref{m1-1}, we get
$n(\lambda-\mu) a_{n}(\lambda)a_{n}(\mu)=0$. Thus $a_{n}(\lambda)=0$,
a contradiction. Therefore, $f(\partial,\lambda)=a_{0}(\lambda)+a_{1}(\lambda) \partial.$ Plugging this back to \eqref{m1-1} and comparing the coefficients of $\partial$ in both sides gives $$(\lambda-\mu)a_1(\lambda)a_1(\mu)=(\lambda-\mu)a_1(\lambda+\mu).$$ This implies $a_1(\lambda)=1.$ Thus $f(\partial,\lambda)=\partial+a_0(\lambda)$.
Substituting this to \eqref{m1-1} gives, after simplification,
\begin{eqnarray}\label{m1-2}
(\lambda-\mu)\big(a_0(\lambda+\mu)+g(\lambda+\mu)\big)=\lambda a_0(\lambda)-\mu a_0(\mu).
\end{eqnarray}
Putting $\mu=0$ in \eqref{m1-2}, we get $g(\lambda)=0$. Thus \eqref{m1-2} becomes
\begin{eqnarray}\label{m1-3}
(\lambda-\mu)a_0(\lambda+\mu)=\lambda a_0(\lambda)-\mu a_0(\mu),
\end{eqnarray}
which implies ${\rm deg\,} a_0(\lambda)\leq 1$ and thus $a_0(\lambda)=\Delta \lambda+\alpha$ for some $\Delta,\alpha\in\C.$
Hence $f(\partial,\lambda)=\partial+\alpha+\Delta \lambda$ and $g(\partial,\lambda)=0$. This proves (1).

(2) In the case $\Delta\neq 0$, assume that $V$ is a nonzero submodule of $V(\alpha,\Delta)$. Then
there exists a nonzero element $u=f(\pa)v\in V$ for some $0\neq f(\pa)\in\C[\pa]$.
If ${\rm deg}\, f(\pa)=0$, then $v\in V$ and thus $V=V(\alpha,\Delta)$.
If ${\rm deg}\, f(\pa)=m>0$, then write $f(\partial)=a_m\partial^m+ \cdots+a_{0} \in \mathbb{C}[\partial],$ with $a_{m} \neq 0$. Thus
\begin{eqnarray}
L_{\la}u=f(\pa+\la)(L_{\la}v)=(\partial+\alpha+\Delta \lambda)f(\pa+\la)v=a_m \Delta v \lambda^{m+1}+\cdots+ {\text {lower\ terms \ of}} \ \lambda \in V[\la],
\end{eqnarray}
which gives $a_m \Delta v\in V$ and thus $v\in V$ since $a_{m} \neq 0$ and $\Delta\neq 0$. It follows $V=V(\alpha,\Delta)$. Hence   $V(\alpha,\Delta)$ is irreducible.

Conversely, it is straightforward to check that $(\partial +\alpha)V(\alpha,0)$ is a nontrivial submodule of $V(\alpha,0)$. Thus $\Delta\neq 0$ when $V(\alpha,\Delta)$ is irreducible.

Note that $(\partial +\alpha)V(\alpha,0)=\C[\partial](\partial +\alpha)v$. One has
$$L_{\lambda}((\partial+\alpha) v)=(\partial+\alpha+\lambda)(\partial+\alpha) v,\ I_{\lambda}((\partial+\alpha) v)=0.$$
Thus the map $(\partial+\alpha) V_{0, \alpha} \rightarrow V_{1, \alpha},(\partial+\alpha) v \mapsto v$ is an isomorphism of $\mathcal{R}$-modules. This proves (2).
\end{proof}

By definition, the annihilation algebra of $\mathcal{R}$
is $\textit{Lie}(\mathcal{R})^+= {\rm span}_{\C }\{L_m,I_m|m\geq -1\}$
with the following relations
\begin{eqnarray}\label{LB}
[L_m,L_{n}]=(m-n)(L_{m+n}+I_{m+n}),\ \
[L_m,I_n]=-(n+1)I_{m+n},\ \
[I_m,I_n]=0,\ \forall\ m,n\in\Z.
\end{eqnarray}
The extended annihilation algebra $\textit{Lie}(\mathcal{R})^e=\C \pa \bigoplus\textit{Lie}(\mathcal{R})^+ $ with relations in \eqref{LB} and
\begin{equation}\label{LB2}
[\partial, L_m]=-(m+1)L_{m-1}, \
[\partial, I_n]=-(n+1)I_{n-1}, \ \forall\ m,n\in\Z.
\end{equation}
For simplicity, we denote $\LL=\textit{Lie}(\mathcal{R})^e$.
Set $\mathcal{L}_n=\bigoplus\limits_{i\geqslant n}(\C L_{i}\oplus\C I_i)$ for $n \geq -1$.
Then we obtain the following filtration:
\begin{equation*}
\mathcal{L}\supset\mathcal{L}_{-1}\supset\mathcal{L}_0\cdots\supset\mathcal{L}_n\supset\cdots.
\end{equation*}
Obviously, it satisfies that $\mathcal{L}_{-1}=\textit{Lie}(\mathcal{R})^+$, $[\mathcal{L}_0, \mathcal{L}_0]=\C I_0+\mathcal{L}_1$, $[\pa,\mathcal{L}_n]=\mathcal{L}_{n-1}$ for $n\geq0$, and $\LL_N$ is
an ideal of $\LL_0$ for any $N> 0$.
\begin{lemm}\label{E1} For any fixed positive integer $N$,
$\LL_0/\LL_N$ is a finite-dimensional solvable Lie algebra.
\end{lemm}
\begin{proof} It is obvious to see that $\LL_0/\LL_N$ is of finite dimension. Consider the derived subalgebras of $\LL_0$:
$\LL_0^{(0)}=\LL_0$ and $\LL_0^{(n+1)}=[\LL_0^{(n)}, \LL_0^{(n)}]$ for $n\geqslant 0$. By the facts that $\LL_0^{(1)}=\C I_0+\mathcal{L}_1\subseteq \mathcal{L}_0$, $[I_{0},\mathcal{L}_1]\subseteq \mathcal{L}_1$, we obtain $\LL_0^{(2)}\subseteq \mathcal{L}_1$. By induction on $n$, one can obtain that $\LL_0^{(n+1)}\subset\LL_n$ for all $n\geqslant 0$. This implies the solvability property of $\LL_0/\LL_N$.
\end{proof}
\begin{theo} \label{main} Let $V$ be a finite nontrivial irreducible conformal
 $\mathcal{R}$-module. There exist some $\alpha\in\C$ and $\Delta \in\C^*$, such that 
\begin{eqnarray*}
V\cong V(\alpha,\Delta)=\mathbb{C}[\partial]v_\Delta,\ L_\lambda
v_\Delta=(\partial+\alpha+\Delta \lambda)v_\Delta, \ I_\lambda v_\Delta=0.
\end{eqnarray*}
\end{theo}
\begin{proof}
By \eqref{conformal}, the conformal $\mathcal{R}$-module $V$ can be viewed as a module over $\LL$ satisfying
\begin{eqnarray}
\LL_n v=0, \ \  for \ v\in V,\ \  n\gg 0.
\end{eqnarray}
Set $V_n=\{v\in V\,|\,\LL_n v=0\}$. Then $V_n\neq \{0\}$ for $n\gg 0$.
Let $N$ be the smallest integer such that $V_N\neq\{0\}$. We have $N\geq 0$ [if not, $V$ is a trivial $\LL$-module].
Note that $I_{-1}$ is in the center of $\mathcal{L}$. By Schur's Lemma, there exists some $\beta\in\C$ such that $I_{-1} v=\beta v$ for all $v\in V$.

Suppose that $N=0$. Thus $\LL_0$ acts trivially on $V_0$. Let $u$ be a nonzero element in $V_0$. Lemma \ref{I} implies that $w_j=\partial^j u\neq0$, $\forall j\geq 0$. Since $V$ is a free $\C[\partial]$-module of finite rank, there exist $f_{1}(\partial), \ldots, f_{m}(\partial) \in \mathbb{C}[\partial]$ such that
\begin{eqnarray}\label{&&}
f_{1}(\partial) w_{1}+\cdots \cdots +f_{m}(\partial) w_{m}=0.
\end{eqnarray}
Denote $n=\max \left\{\operatorname{deg} f_{1}(\partial)+1, \ldots, \operatorname{deg} f_{m}(\partial)+m\right\}$ .
Using $ R_{\partial}=L_{\partial}-\mathrm{ad}_{\partial}$ and the binomial formula, we have, for $k \leq n+1$,
 \begin{eqnarray}
I_{n} \partial^{k}=R_{\partial}^k I_{n}= \left(L_{\partial}-\mathrm{ad}_{\partial}\right)^{k} I_{n}=\sum_{i=0}^{k} \frac{k!(n+1)!}{(k-i)!i!(n+1-i)!} \partial^{k-i} I_{n-i}. \end{eqnarray}
 Recall that $I_{-1}u=\beta u$ and $I_k u=0$ for $k\geq0$. By the action of $I_{n-1}$ on \eqref{&&}, we have $I_{-1}u=\beta u=0$ and thus $\beta=0$.
In this case, one can check that
\begin{eqnarray*}
(L_{-1}-\partial)X v=X (L_{-1}-\partial)v, \ \forall \ X\in\mathcal{L}, v\in V.
\end{eqnarray*}
By Schur's Lemma again, there exists some $\a\in\C$ such that $(L_{-1}-\partial) v=\a v$ for all $v\in V$. It follows that $\C[\pa]u$ is a submodule of $V$. Hence $V=\C[\pa]u$ by the irreducibility of $V$. The $\la$-action of $\mathcal
{R}$ on $V=\C[\pa]u$ is given by $I_\la u=0$ and
\begin{eqnarray*}
L_\la u=\sum_{j\in\Z^{+}}(L_{(j)}u)\frac{\lambda^j}{j!}=\sum_{j\in\Z^{+}}(L_{j-1} u)\frac{\lambda^j}{j!}=L_{-1} u=(\partial+\alpha)u.
\end{eqnarray*}
But then Proposition \ref{key1} implies that $V$ is reducible. This is a contradiction.

Now we have $N\geqslant 1$. By \cite[Lemma 3.1]{CK}, $V_N$ is a finite-dimensional vector space.
Note that $V_N$ is actually an $\LL_0/\LL_N$-module. By Lemma \ref{E1},
there exists a nonzero common eigenvector $u\in V_N$ under the action of $\LL_0/\LL_N$, and then the action of $\LL_0$.
Namely, there exists a linear function $\chi$ on $\LL_0$ such that $x u=\chi(x) u$ for any $x\in\LL_0$. Since $[\LL_0, \LL_0]=\C I_{0}\bigoplus \LL_1$, $I_0 u=\LL_1 u=0$.
Assume that $\chi(L_0)=\Delta$ for some $\Delta\in\C$. We have $\Delta\neq 0$ [if not, $N=0$].
With a similar discussion as in the case $N=0$, we have $I_{-1}u=0$ and there is some $\a\in\C$, such that
\begin{eqnarray}
L_{-1} v=(\pa+\a) v,\  \forall\ v\in V.
\end{eqnarray}
Therefore, $\C[\partial]u$ is a nonzero submodule of $V$. It follows $V=\C[\partial]u$. The $\la$-action of $\mathcal{R}$ on $V$ is given by $I_\la u=0$ and
\begin{eqnarray*}
L_\la u=\sum_{j\in\Z^{+}}(L_{(j)}u)\frac{\lambda^j}{j!}=\sum_{j\in\Z^{+}}(L_{j-1} u)\frac{\lambda^j}{j!}=L_{-1} u+(L_0 u)\la=(\partial+\alpha+\Delta\lambda)u.
\end{eqnarray*}
By Proposition \ref{key1}, $V\cong V(\alpha,\Delta)$ with $\Delta\neq0$. The proof is finished.
\end{proof}

\begin{rem} {\rm The modules $V(\alpha,\Delta)$ with
$\Delta\neq 0$ exhaust all finite non-1-dimensional irreducible
conformal modules over $\mathcal{R}$. Therefore, $V(\alpha,\Delta)$ with $\Delta\neq 0$, along with the one-dimensional modules $\C c_\gamma$ with $\gamma\in\C^*$, form a complete list of finite irreducible conformal $\mathcal{R}$-modules.}
\end{rem}

\section{Extensions of conformal $\mathcal{R}$-modules}

In this section, we study extensions between two finite irreducible conformal $\mathcal{R}$-modules. Let $M$ be any irreducible conformal $\mathcal{R}$-module. By Definition \ref{def1}, an $\mathcal{R}$-module structure on $M$ is given by $L_\lambda, I_\mu \in {\rm End}_\C(M)[\lambda]$ such that
\begin{eqnarray}
&&[L_\lambda, L_\mu]=(\lambda-\mu)(L_{\lambda+\mu}+I_{\lambda+\mu}), \label{L}\\
&&[L_\lambda, I_\mu]=-\mu I_{\lambda+\mu},\label{LH}\\
&&[I_\lambda, L_\mu]=\lambda I_{\lambda+\mu},\label{HL}\\
&&[\partial, L_\lambda]=-\lambda L_\lambda ,\label{a}\\
&&[\partial, I_\lambda]=-\lambda I_\lambda,\label{b}\\
&&[I_\lambda, I_\mu]=0. \label{H}
\end{eqnarray}

\subsection{Extensions involving 1-dimensional modules}

First, we consider extensions of $\mathcal{R}$-modules of the form
\begin{eqnarray}\label{w1m}
0\longrightarrow \C{c_\gamma}\longrightarrow E \longrightarrow V(\alpha,\Delta) \longrightarrow 0.
\end{eqnarray}
 As a module over $\C[\partial]$, $E$ in \eqref{w1m} is isomorphic to $\C {c_\gamma}\oplus V(\alpha,\Delta)$, where $\C {c_\gamma}$ is an $\mathcal{R}$-submodule, and $V(\alpha,\Delta)=\C[\partial]v_\Delta$ such that the following identities hold in $E$:
\begin{eqnarray}\label{cm1}
L_\lambda v_\Delta=(\partial+\alpha+\Delta\lambda)v_\Delta+f(\lambda)c_\gamma,\ I_\lambda v_\Delta=g(\lambda)c_\gamma,\ \text {for some}\  f(\lambda),\,g(\lambda)\in\C[\lambda].
\end{eqnarray}

We first give the formula of the trivial cocycles.

\begin{lemm}\label{lem1} All trivial extensions of the form \eqref{w1m} are given by \eqref{cm1}, where $f(\lambda)$ is a scalar multiple of $\alpha+\gamma+\Delta\lambda$, and $g(\lambda)=0$.
\end{lemm}
\begin{proof} Suppose that \eqref{w1m} represents a trivial cocycle. This means that the exact sequence
\eqref{w1m} is split
and hence there exists $v_\Delta'=\varphi(\partial)v_\Delta+a c_\gamma\in E$, where $a\in\C$, such that
\begin{eqnarray*}
L_\lambda v_\Delta'=(\partial+\alpha+\Delta\lambda)v_\Delta'
=(\partial+\alpha+\Delta\lambda)\varphi(\partial)v_\Delta+a(\gamma+\alpha+\Delta\lambda)c_\gamma,
\end{eqnarray*}
and $I_\lambda v_\Delta'=0.$ On the other hand, it follows from \eqref{cm1} that
\begin{eqnarray*}
L_\lambda v_\Delta'&=&\varphi(\partial+\lambda)(\partial+\alpha+\Delta\lambda)v_\Delta+\varphi(\partial+\lambda)f(\lambda)c_\gamma,\\
I_\lambda v_\Delta'&=&\varphi(\partial+\lambda)g(\lambda)c_\gamma.
\end{eqnarray*}
Comparing both expressions for $L_\lambda v_\Delta'$ and $I_\lambda v_\Delta'$  respectively, we see that $\varphi(\partial)$ is a nonzero constant. Hence $f(\lambda)$ is a scalar multiple of $\alpha+\gamma+\Delta\lambda$ and $g(\lambda)=0$.
\end{proof}

Applying both sides of \eqref{L} and \eqref{LH} to $v_\Delta$, we obtain the following equations
\begin{eqnarray}
(\alpha+\gamma+\lambda+\Delta\mu)f(\lambda)-(\alpha+\gamma+\mu+\Delta\lambda)f(\mu)&=&(\lambda-\mu)\big(f(\lambda+\mu)+g(\lambda+\mu)\big),\label{l1}\\
(\alpha+\gamma+\mu+\Delta\lambda)g(\mu)&=&\mu g(\lambda+\mu).\label{lh1}
\end{eqnarray}
Therefore, solving the extension problem above is equivalent to solving the functional equations \eqref{l1} and \eqref{lh1} on the polynomials $f(\la)$ and $g(\la)$.
Letting $\lambda=0$ in \eqref{lh1}, we obtain
\begin{eqnarray}\label{I1}
(\alpha+\gamma)g(\mu)=0.
\end{eqnarray}

\textbf{Case 1.} $\alpha+\gamma\neq0$.

In this case, \eqref{I1} gives $g(\mu)=0$. Setting $\mu=0$ in \eqref{l1} gives $f(\lambda)=\frac{f(0)}{\alpha+\gamma}(\alpha+\gamma+\Delta\lambda)$. Then the corresponding extension is trivial by Lemma \ref{lem1}.

\textbf{Case 2.} $\alpha+\gamma=0$.

In this case, \eqref{l1} and \eqref{lh1} respectively reduce to
\begin{eqnarray}
(\lambda+\Delta\mu)f(\lambda)-(\mu+\Delta\lambda)f(\mu)&=&(\lambda-\mu)\big(f(\lambda+\mu)+g(\lambda+\mu)\big),\label{l11}\\
(\mu+\Delta\lambda)g(\mu)&=&\mu g(\lambda+\mu).\label{lh11}
\end{eqnarray}
Substituting $g(\lambda)=\sum_{i=0}^m a_i\lambda^i$ into \eqref{lh11} and comparing the coefficients of $\la^m$, we obtain that $m\leq1$. Then it is easy to check that $g(\lambda)=a_1 \lambda$ if $\Delta=1$, or else $g(\lambda)=0$. Putting $\mu=0$ in \eqref{l11} with $\Delta=1$, we get $a_1=f(0)=0$. Thus $g(\lambda)\equiv0$.

 Now \eqref{l11} becomes
\begin{eqnarray}
(\lambda+\Delta\mu)f(\lambda)-(\mu+\Delta\lambda)f(\mu)=(\lambda-\mu)f(\lambda+\mu).\label{l11-1}
\end{eqnarray}
 Assume that $f(\lambda)=\sum_{i=0}^n b_i\lambda^i$ with $b_n\neq0$. Plugging this into \eqref{l11-1} gives
 \begin{eqnarray}
(\la+\Delta\mu)\sum^{n}_{i=0}b_i\lambda^{i}-(\mu+\Delta\lambda)\sum^{n}_{i=0}b_i\mu^{i}=(\lambda-\mu)\sum^{n}_{i=0}b_i(\lambda+\mu)^{i}\label{2h1}.
\end{eqnarray}
Assuming $n\geqslant4$ and comparing the coefficients of $\mu^2\la^{n-1}$ in \eqref{2h1}, we have
\begin{eqnarray*}
b_n\big(\mbox{$\binom{n}{2}-\binom{n}{1}$}\big)=0,
\end{eqnarray*}
which gives $n=3$. This contradicts the assumption that $n\geqslant4$. Thus $n\leq 3.$ Then solutions of \eqref{l11-1} of degree less than or equal to 3 can be directly checked. We have the following:
\begin{lemm}\label{lem1-1} All solutions of \eqref{l11-1} are as follows:
\begin{itemize}
\item[{\rm(1)}] $f(\lambda)=0$ for $\Delta$ arbitrary.
\item[{\rm (2)}] $f(\lambda)=b_1\lambda$ for $\Delta$ arbitrary.
\item[{\rm (3)}] $f(\lambda)=b_2\lambda^2$ for $\Delta=1$.
\item[{\rm (4)}] $f(\lambda)=b_3\lambda^3$ for $\Delta=2$.

\end{itemize}
\end{lemm}

By Lemmas \ref{lem1} and \ref{lem1-1}, and the  above discussion, we have
\begin{theo}\label{p2} Nontrivial extensions of the form \eqref{w1m} exist if and only if $\alpha+\gamma=0$ and $\Delta$ is either 1 or 2. In these cases, they are given by \eqref{cm1}, where
\begin{itemize}
\item[{\rm (i)}] $\Delta=1$, $g(\lambda)=0$ and $f(\lambda)=b_2\lambda^2$, with $b_2\neq0$.
\item[{\rm (ii)}] $\Delta=2$, $g(\lambda)=0$ and $f(\lambda)=b_3\lambda^3$, with $b_3\neq0$.
\end{itemize}
The corresponding spaces
${\rm Ext}(V(\alpha,1),\C{c_{-\alpha}})$ and ${\rm Ext}(V(\alpha,2),\C{c_{-\alpha}})$ are 1-dimensional.
\end{theo}

Next we consider extensions of $\mathcal{R}$-modules of the form
\begin{eqnarray}\label{w2m}
0\longrightarrow V(\alpha,\Delta)\longrightarrow E \longrightarrow \C c_\gamma \longrightarrow 0.
\end{eqnarray}
As a vector space, $E$ in \eqref{w2m} is isomorphic to $V(\alpha,\Delta)\oplus\C {c_\gamma}$. Here $V(\alpha,\Delta)=\C[\partial]v_\Delta$ is an $\mathcal{R}$}-submodule and
we have
\begin{eqnarray}\label{cm2}
L_\lambda c_\gamma=f(\partial,\lambda)v_\Delta,\ \
I_\lambda c_\gamma=g(\partial,\lambda)v_\Delta,\ \
\partial c_\gamma=\gamma c_\gamma+a(\partial)v_\Delta,
\end{eqnarray}
for some $f(\partial,\lambda),\ g(\partial,\lambda)\in\C[\partial,\lambda],$ and $a(\partial)\in\C[\partial]$.

\begin{lemm}\label{lem2} All trivial extensions of the form \eqref{w2m} are given by \eqref{cm2} with
$f(\partial,\lambda)=(\pa+\alpha+\Delta\lambda)\phi(\partial+\lambda)$, $g(\partial,\lambda)=0$ and $a(\partial)=(\partial-\gamma)\phi(\partial)$, where $\phi$ is a polynomial.
\end{lemm}
\begin{proof}  Suppose that \eqref{w2m} represents a trivial cocycle. This means that the exact sequence
\eqref{w2m} is split
and hence there exists $c'_\gamma=b c_\gamma+\phi(\partial)v_\Delta\in E$, where $b\in\C$, such that $L_\lambda c'_\gamma=I_\lambda c'_\gamma=0$ and $\partial c'_\gamma=\gamma c'_\gamma$. On the other hand, a short computation shows that
\begin{eqnarray*}
L_\lambda c'_\gamma&=&(\partial+\alpha+\Delta\lambda)\phi(\partial+\lambda)v_\Delta+b f(\partial,\lambda)v_\Delta,\\
I_\lambda c'_\gamma&=& b g(\partial,\lambda)v_\Delta,\
\partial c'_\gamma= b\gamma c_\gamma+\big(  b a(\partial)+\partial\phi(\partial)\big)v_\Delta.
\end{eqnarray*}
Comparing both expressions for $L_\lambda c'_\gamma$, $I_\lambda c'_\gamma$ and $\partial c'_\gamma$ respectively, we obtain the result.
\end{proof}

\begin{theo}\label{theo2}
There are nontrivial extensions of $\mathcal{R}$-modules of the form \eqref{w2m} if and only if $\alpha+\gamma=0$ and $\Delta=1$. In this case, ${\rm dim}_\C{{\rm Ext}}\big(\C{c_{-\alpha}}, (V(\alpha,1))\big)=1,$  and the unique (up to a scalar) nontrivial extension is given by
$$L_\lambda c_\gamma=c v_\Delta,\ I_\lambda c_\gamma=0,\
\partial c_\gamma=\gamma c_\gamma+c v_\Delta,$$
where $c\in\C^*$.
\end{theo}

\begin{proof}

Applying both sides of \eqref{L}, \eqref{a} and \eqref{b} to $c_\gamma$ gives the following functional equations:
\begin{eqnarray}
(\partial +\alpha +\Delta \lambda )f(\partial +\lambda ,\mu )-(\partial +\alpha +\Delta \mu )f(\pa +\mu ,\lambda )&=&(\lambda -\mu )\big(f(\partial ,\lambda +\mu )+g(\partial ,\lambda +\mu )\big),\nonumber\\&&\quad\quad\quad\quad\quad\quad\quad\quad\quad\quad \label{s5}\\
(\partial +\lambda-\gamma)f(\partial ,\lambda )&=&(\partial +\alpha +\Delta \lambda )a(\partial +\lambda ),\label{s6}\\
(\partial +\lambda-\gamma)g(\partial ,\lambda )&=&0.\label{s7}
\end{eqnarray}
Obviously, $g(\partial ,\lambda )=0$ by \eqref{s7}. Replacing $\partial$ by $\bar\partial=\partial+\alpha$ and letting $\bar f(\bar\partial, \lambda) = f(\bar\partial-\alpha, \lambda)$, and $\bar a(\bar\partial) = a(\bar\partial-\alpha)$, we can rewrite
\eqref{s5} and \eqref{s6} in a homogeneous form
\begin{eqnarray}
(\bar\partial+\Delta\lambda)\bar f(\bar\partial+\lambda ,\mu )-(\bar\partial +\Delta \mu )\bar f(\bar\partial+\mu ,\lambda )&=&(\lambda -\mu )\bar f(\bar \partial ,\lambda +\mu ),\label{s9}\\
(\bar\partial-\alpha +\lambda-\gamma)\bar f(\bar \partial ,\lambda )&=&(\bar\partial +\Delta \lambda )\bar a(\bar\partial +\lambda ).\label{s10}
\end{eqnarray}
Taking $\mu=0$ in \eqref{s9}, we can obtain that, if degree of $\bar f$ is positive, $\bar f(\bar\partial, \lambda)$ is a scalar multiple of $(\bar\partial+\Delta\lambda)\bar f(\bar \partial+\lambda,0)$, where $\bar f(\bar\partial+\lambda,0)$ is a polynomial in $\bar \partial+\lambda$, or else $\bar f(\bar\partial,\lambda)=c\in\C.$

Assume that $\bar f(\bar\partial, \lambda)$ is a scalar multiple of $(\bar\partial+\Delta\lambda)\bar f(\bar \partial+\lambda,0)$. Letting $\lambda=0$ in \eqref{s10}, we have $\bar a(\bar\partial)$ is a scalar multiple of $(\bar\partial-\alpha-\gamma)\bar f(\bar\partial,0)$. By Lemma \ref{lem2} and noting that we have employed a shift by $\alpha$, the corresponding extension is trivial.

It is left to consider the case $\bar f(\bar \partial,\lambda)=c$. Substituting this into \eqref{s10} gives $\bar a(\bar \partial)=\bar f(\bar \partial,\lambda)=c$, and in particular, $\alpha+\gamma=0,$ $\Delta=1$ if $c\neq 0$. The proof is finished.
\end{proof}

\subsection{Extensions of conformal $\mathcal{R}$-modules}
In this subsection, we study extensions of $\mathcal{R}$-modules of the form
\begin{eqnarray}\label{3m}
0\longrightarrow V(\bar\alpha,\bar\Delta)\longrightarrow E \longrightarrow V(\alpha,\Delta) \longrightarrow 0,
\end{eqnarray}
where $V(\bar\alpha,\bar\Delta)=\C[\partial]v_{\bar\Delta}$, $V(\alpha,\Delta)=\C[\partial]v_\Delta$, with $\Delta,\ \bar\Delta\in\C^*$.

As a $\C[\partial]$-module, $E\cong V(\bar\alpha,\bar\Delta)\oplus V(\alpha,\Delta)$. Here $V(\bar\alpha,\bar\Delta)=\C[\partial]v_{\bar\Delta}$ is an $\mathcal{R}$-submodule and we have
\begin{eqnarray}\label{3m*}
L_\lambda v_\Delta=(\partial+\alpha+\Delta\lambda)v_\Delta+f(\partial,\lambda)v_{\bar\Delta},\ I_\lambda v_\Delta=g(\partial,\lambda)v_{\bar\Delta},
\end{eqnarray}
where $f(\partial,\lambda)$ and $ g(\partial,\lambda)$ are polynomials in  $\partial$ and $\lambda$.
\begin{lemm}\label{lem4} All trivial extensions of the form \eqref{3m} are given by \eqref{3m*}, where
$f(\partial,\lambda)$ is a scalar multiple of $(\partial+\alpha+\Delta\lambda)\phi(\partial)-(\partial+\bar\alpha+\bar\Delta\lambda)\phi(\partial+\lambda)$ and $g(\partial,\lambda)=0$, with $\phi$ a polynomial.
\end{lemm}
\begin{proof} Suppose that \eqref{3m} represents a trivial cocycle. This means that the exact sequence
\eqref{3m} is split
and hence there exists $v_\Delta'=\varphi (\partial)v_\Delta+\phi(\partial)v_{\bar\Delta}\in E$, such that
\begin{eqnarray*}
L_\lambda v_\Delta'=(\partial+\alpha+\Delta\lambda)v_\Delta'
=(\partial+\alpha+\Delta\lambda)\big(\varphi(\partial)v_\Delta+\phi(\partial)v_{\bar\Delta}\big), \ I_\lambda v'_\Delta=0.
\end{eqnarray*}
On the other hand, a short computation shows that
\begin{eqnarray*}
L_\lambda v_\Delta'&=&L_\lambda (\varphi (\partial)v_\Delta+\phi(\partial)v_{\bar\Delta})\\
&=&\varphi(\partial+\lambda)L_\lambda v_\Delta+\phi(\partial+\lambda)L_\lambda v_{\bar\Delta}\\
&=&(\partial+\alpha+\Delta\lambda)\varphi(\partial+\lambda)v_\Delta+\big(\varphi(\partial+\lambda)f(\partial,\lambda)+\phi(\partial+\lambda)(\partial+\bar\alpha+\bar\Delta\lambda)\big)v_{\bar\Delta}.
\end{eqnarray*}
Comparing both expressions for $L_\lambda v_\Delta'$ gives
\begin{eqnarray*}
&&(\partial+\alpha+\Delta\lambda)\varphi(\partial)=(\partial+\alpha+\Delta\lambda)\varphi(\partial+\lambda),\\
&&(\partial+\alpha+\Delta\lambda)
\phi(\partial)=f(\partial,\lambda)\varphi(\partial+\lambda)+(\partial+\bar\alpha+\bar\Delta\lambda)\phi(\partial+\lambda).
\end{eqnarray*}
The former equation implies that $\varphi(\partial)$ is a nonzero complex number $c$, whereas the latter shows that $f(\partial,\lambda)$ is a scalar multiple of $(\partial+\alpha+\Delta\lambda)\phi(\partial)-(\partial+\bar\alpha+\bar\Delta\lambda)\phi(\partial+\lambda)$.
Similarly, we have
\begin{eqnarray*}
I_\lambda v_\Delta'=c g(\partial,\lambda)v_{\bar\Delta}=0,
\end{eqnarray*}
which gives $g(\partial,\lambda)=0$ since $c\neq 0$.
\end{proof}

Applying both sides of \eqref{L} and \eqref{LH} to $v_\Delta$ gives the following functional equations:
\begin{eqnarray}
(\lambda-\mu)\big(f(\partial,\lambda+\mu)+g(\partial,\lambda+\mu)\big)&=&(\partial+\lambda+\Delta\mu+\alpha)f(\partial,\lambda)+(\partial+\bar\Delta\lambda+\bar\alpha)f(\partial+\lambda,\mu)\nonumber \\&& -(\partial+\mu+\Delta\lambda+\alpha)f(\partial,\mu) -(\partial+\bar\Delta\mu+\bar\alpha)f(\partial+\mu,\lambda),\nonumber\\&&\quad\quad\quad\quad\quad\quad\quad\quad\quad\quad \label{wl}\\
-\mu g(\partial,\lambda+\mu)&=&(\partial+\bar\Delta\lambda+\bar\alpha)g(\partial+\lambda,\mu)-(\partial+\mu+\Delta\lambda+\alpha)g(\partial,\mu).\label{wlh}
\quad\end{eqnarray}
Putting $\lambda=0$ in \eqref{wlh} gives that $(\alpha-\bar\alpha)g(\partial,\mu)=0$. Thus $g(\partial,\mu)=0$ if $\alpha\neq\bar\alpha$. With this and setting $\lambda=0$ in
\eqref{wl} we obtain
\begin{eqnarray}\label{14*}
f(\partial,\mu)=\frac1{\alpha-\bar\alpha}\big((\partial+\Delta\mu+\alpha)\tilde f(\partial)-(\partial+\bar\Delta\mu+\bar\alpha)\tilde f(\partial+\mu)\big),
\end{eqnarray}
where $\tilde f(\partial)=f(\partial,0)$. By Lemma \ref{lem4}, the corresponding extension is trivial.

In the following we always assume that $\alpha=\bar\alpha$. For convenience,
put $\bar\partial=\partial+\alpha$, $\bar f(\bar\partial,\lambda)=f(\bar\partial-\alpha,\lambda),$ $\bar g(\bar\partial,\lambda)=g(\bar\partial-\alpha,\lambda)$. However, we will continue to write $\partial$ for $\bar\partial$, $f$ for $\bar f$ and $g$ for $\bar g$. Then the functional equations \eqref{wl} and \eqref{wlh} become
\begin{eqnarray}
(\lambda-\mu)\big(f(\partial,\lambda+\mu)+g(\partial,\lambda+\mu)\big)&=&(\partial+\lambda+\Delta\mu)f(\partial,\lambda)+(\partial+\bar\Delta\lambda)f(\partial+\lambda,\mu)\nonumber \\
&&-(\partial+\mu+\Delta\lambda)f(\partial,\mu)-(\partial+\bar\Delta\mu)f(\partial+\mu,\lambda),\label{L3}\\
-\mu g(\partial,\lambda+\mu)&=&(\partial+\bar\Delta\lambda)g(\partial+\lambda,\mu)-(\partial+\mu+\Delta\lambda)g(\partial,\mu).\label{H3}
\end{eqnarray}
Note that if $g(\pa,\la)=0$, then $f(\pa,\la)$ is completely determined by
\begin{eqnarray}
(\lambda-\mu)f(\partial,\lambda+\mu)&=&(\partial+\lambda+\Delta\mu)f(\partial,\lambda)+(\partial+\bar\Delta\lambda)f(\partial+\lambda,\mu)\nonumber \\&&-(\partial+\mu+\Delta\lambda)f(\partial,\mu)-(\partial+\bar\Delta\mu)f(\partial+\mu,\lambda).\label{L3-3}
\end{eqnarray}
All solutions of \eqref{L3-3}, corresponding to nontrivial cocycles, were solved in \cite{CKW1} (see also Theorem \ref{th4}).
The following result is due to \cite[Lemma 3.2]{CKW1}.
\begin{lemm}{\label{important}} Let $f(\partial,\lambda)=\sum_{i=0}^m b_i\pa^{m-i}\la^i$ be a solution of \eqref{L3-3}.
\begin{itemize}
\item[{\rm(1)}]
If $m\geq 2$ and $\Delta-\bar\Delta\neq m-1$, then $f(\partial,\lambda)$ is a scalar multiple of
$(\partial+\bar\Delta\lambda)(\partial+\lambda)^{m-1}-(\partial+\Delta\lambda)\partial^{m-1}.$
\item[{\rm(2)}]If $\Delta-\bar\Delta= m-1$ and $m\geq 3$, then $b_1=0$.
\end{itemize}
\end{lemm}

We may assume from now on that $g(\pa,\la)\neq0$.
By the nature of \eqref{H3}, we may assume that a solution is a homogeneous polynomial in $\pa$ and $\la$ of degree $n$. Hence we can write $g(\pa,\la)=\sum_{i=0}^n a_i\pa^{n-i}\la^i,$  where $ a_i\in\C$.
Substituting this into \eqref{H3} gives
\begin{eqnarray}\label{co1}
-\mu\sum^{n}_{i=0}a_i\partial^{n-i}(\lambda+\mu)^i=(\partial+\bar\Delta\lambda)\sum^{n}_{i=0}a_i(\partial+\lambda)^{n-i}\mu^i-(\partial+\mu+\Delta\lambda)\sum^{n}_{i=0}a_i\partial^{n-i}\mu^i.
\end{eqnarray}
 Comparing the coefficients of $\la\mu^n$ in both sides of \eqref{co1} gives $\Delta-\bar\Delta=n$.

Assume that $n\geqslant4$.
Comparing the coefficients of $\la^{n+1}$ in \eqref{co1} gives $a_0=0$ since $\bar\Delta\neq0$. Therefore $g(\pa,0)=0$. This means that $g(\pa,\la)=\la \tilde g(\pa,\la)$, where $\tilde g(\pa,\la)=\sum_{i=1}^{n}{a_i}\pa^{n-i}\la^{i-1}$. Then \eqref{co1} amounts to
\begin{eqnarray}\label{co2}
\sum^{n}_{i=1}a_i\partial^{n-i}(\lambda+\mu)^{i}=(\partial+\mu+\Delta\lambda)\sum^{n}_{i=1}a_i\partial^{n-i}\mu^{i-1}-(\partial+\bar\Delta\lambda)\sum^{n}_{i=1}a_i(\partial+\lambda)^{n-i}\mu^{i-1}.
\end{eqnarray}
Setting $\mu=0$ in \eqref{co2} gives
\begin{eqnarray}\label{co3}
\sum^{n}_{i=1}a_i\partial^{n-i}\lambda^i=(\partial+\Delta\lambda)a_1\partial^{n-1}-(\partial+\bar\Delta\lambda)a_1(\partial+\lambda)^{n-1}.
\end{eqnarray}
We see that $a_1\neq0$ [if not, $a_i=0$ for $i=0, 1,\cdots,n$ and thus $g(\pa,\la)=0$]. Equating the coefficients of $\partial^{n-i}\lambda^i$ in \eqref{co2}, we get
\begin{eqnarray}\label{co4}
a_i=-a_1\binom{n-1}{i}-a_1\bar\Delta\binom{n-1}{i-1},\ \ 2\leqslant i \leqslant n.
\end{eqnarray}
Comparing the coefficients of $\lambda^{n-1}\mu$ in \eqref{co2} gives
$na_n=-\bar\Delta a_2$ and hence by \eqref{co4},
\begin{eqnarray}\label{co5}
n=-\mbox{$\binom{n-1}{2}$}-(n-1)\bar\Delta.
\end{eqnarray}
Similarly, comparing the coefficients of $\partial\lambda^{n-2}\mu$ in \eqref{co2} and using \eqref{co4} again, we have
\begin{eqnarray}\label{co6}
(n-1)\big(1+(n-1)\bar\Delta\big)=-\big(1+(n-2)\bar\Delta\big)\Big(\mbox{$\binom{n-1}{2}$}+(n-1)\bar\Delta\Big).
\end{eqnarray}
Combining \eqref{co5} with \eqref{co6} yields $\bar\Delta=1$. Then \eqref{co5} becomes $\mbox{$\binom{n-1}{2}$}+2n-1=0$, which certainly cannot happen because $n$ would not be an integer. Therefore $n$ can be at most $3$, and we have only the following four possible cases to discuss.
\vskip5pt
{\bf Case 1.} $n=0$.
\vskip5pt
In this case, $\Delta=\bar\Delta$ and $g(\pa,\la)=a_0 \in\C^*$. Plugging this into \eqref{L3} gives
\begin{eqnarray}
(\lambda-\mu)\big(f(\partial,\lambda+\mu)+a_0\big)&=&(\partial+\lambda+\Delta\mu)f(\partial,\lambda)+(\partial+\bar\Delta\lambda)f(\partial+\lambda,\mu)\nonumber \\
&&-(\partial+\mu+\Delta\lambda)f(\partial,\mu)-(\partial+\bar\Delta\mu)f(\partial+\mu,\lambda).\label{L4}
\end{eqnarray}
Setting $\mu=0$ in \eqref{L4} gives
\begin{eqnarray*}
a_0\lambda=(\partial+\bar\Delta\lambda)f(\partial+\lambda,0)-(\partial+\Delta\lambda)f(\partial,0),
\end{eqnarray*}
which implies that $f(\partial,0)$ is a constant and thus $a_0=0$. A contradiction. Thus there is no nontrivial solution in this case.
\vskip5pt
{\bf Case 2.} $n=1$.
\vskip5pt
 In this case, $\Delta-\bar\Delta=1$ and $g(\pa,\la)=a_1\la$ with $a_1\in\C^*$. Now \eqref{L3} becomes
\begin{eqnarray}
(\lambda-\mu)\big(f(\partial,\lambda+\mu)+a_1(\lambda+\mu)\big)&=&(\partial+\lambda+\Delta\mu)f(\partial,\lambda)+(\partial+\bar\Delta\lambda)f(\partial+\lambda,\mu)\nonumber \\
&&-(\partial+\mu+\Delta\lambda)f(\partial,\mu)-(\partial+\bar\Delta\mu)f(\partial+\mu,\lambda).\label{L5}
\end{eqnarray}
Let $f(\partial,\lambda)$ be a solution of \eqref{L5}.
Assume that the total degree of  $f(\partial,\lambda)$ in $\partial$ and $\lambda$ is $m$, and write  $f(\partial,\lambda)=\sum_{i=1}^{m}f_i(\partial,\lambda)$, where
$f_i(\partial,\lambda)$ is the $i$th homogenous bifurcation of $f(\partial,\lambda)$. Obviously, $f_1(\partial,\lambda)\neq 0$ and $f_0(\partial,\lambda)=0$.
By \eqref{L3-3} and \eqref{L5}, $f_{i}(\partial,\lambda)$ with $i\geq 2$ may satisfy \eqref{L3-3}. By Lemma \ref{lem4} and the first assertion of Lemma \ref{important}, these $f_{i}(\partial,\lambda)$ for $i\geq 3$ are trivial cocycles.
Write $f_2(\partial,\lambda)=b_{20}\partial^2+b_{21}\partial\lambda+b_{22}\lambda^2$.
Plugging this into \eqref{L3-3} we obtain that $b_{20}=b_{21}=0$. Thus  $f_2(\partial,\lambda)=b_{22}\lambda^2$. By Lemma \ref{lem4} with $\phi(\partial)=\partial$, $f_2(\partial,\lambda)$ is also a trivial cocycle. Thus we may subtract these trivial cocycles and assume that
$f(\partial,\lambda)=f_1(\partial,\lambda)$, which is a homogenous polynomial in $\partial$ and $\lambda$ of degree 1.
Letting $\mu=0$ in \eqref{L5} gives
\begin{eqnarray}
a_1\lambda^{2}=(\partial+\bar\Delta\lambda)f(\pa+\lambda,0)-(\partial+\Delta\lambda)f(\pa,0).
\end{eqnarray}
It follows that $f(\pa,0)=\frac{a_1}{\bar\Delta}\pa$. Thus we can write $f(\partial,\lambda)=\frac{a_1}{\bar\Delta}\pa+b_1\la$.
It is easy to check that it satisfies \eqref{L5}. However, by Lemma \ref{lem4}, $b_1\la$ is a trivial cocycle. Thus we may assume that
$f(\partial,\lambda)=\frac{a_1}{\bar\Delta}\pa$.
The corresponding extension is nontrivial unless $a_1=0$.
\vskip5pt
{\bf Case 3.} $n=2$.
\vskip5pt
In this case, $\Delta-\bar\Delta=2$ and $g(\pa,\mu)=a_1\pa\mu+a_2\mu^2$ with $(a_1,a_2)\neq (0,0)$. Note that \eqref{co4} holds for $n\geq 2$. Thus, by \eqref{co4}, $a_2=-\bar\Delta a_1$ and $g(\pa,\mu)=a_1(\pa-\bar\Delta\mu)\mu$. Substituting this in \eqref{L3} gives
\begin{eqnarray}
&&(\lambda-\mu)\big(f(\partial,\lambda+\mu)+a_1\pa(\lambda+\mu)-a_1\bar\Delta(\la+\mu)^2\big)\nonumber\\
&&\quad\quad\quad\quad\quad\quad\quad\quad\quad=(\partial+\lambda+\Delta\mu)f(\partial,\lambda)+(\partial+\bar\Delta\lambda)f(\partial+\lambda,\mu)\nonumber \\
&&\quad\quad\quad\quad\quad\quad\quad\quad\quad\quad-(\partial+\mu+\Delta\lambda)f(\partial,\mu)-(\partial+\bar\Delta\mu)f(\partial+\mu,\lambda).\label{L6}
\end{eqnarray}
Similarly to the discussion in the case $n=1$, we have the total degree of  $f(\partial,\lambda)$ in $\partial$ and $\lambda$ is at most $3$.
Let $f_i(\partial,\lambda)$ be the $i$th homogenous bifurcation of $f(\partial,\lambda)$, where $0\leq i \leq 3$. By \eqref{L3-3} and \eqref{L6}, each $f_{i}(\partial,\lambda)$ with $i\neq 2$ may be a solution of \eqref{L3-3} and $f_{2}(\partial,\lambda)\neq 0$ since $g(\partial,\lambda)\neq 0$. Then it is easy to check that $f_0(\partial,\lambda)=f_1(\partial,\lambda)=0$ because $\Delta-\bar\Delta=2$.

For $i=2$, write  $f_2(\partial,\lambda)=b_{20}\partial^2+b_{21}\partial\lambda+b_{22}\lambda^2$.
Plugging this into \eqref{L6} gives
$$a_1=-b_{20},\ \bar\Delta=-1,\ \Delta=1,\ b_{22}=a_1+b_{21}.$$
Thus $g(\pa,\la)=a_1(\pa+\la)\la,$ and $f_2(\pa,\la)=-a_1\pa^2+b_{21}\pa\lambda+(a_1+b_{21})\lambda^2$. By Lemma \ref{lem4} with $\phi(\partial)=\partial$, we may subtract from $f_2(\partial,\lambda)$ the trivial cocycle of the form  $b_{21}\big((\partial+\Delta\lambda)\phi(\partial)-(\partial+\bar\Delta\lambda)\phi(\partial+\lambda)\big)$, and thus assume that $f_2(\partial,\lambda)=a_1(\la^2-\partial^2)$.

For $i=3$, we can write
$f_3(\partial,\lambda)=b_{30}\partial^3+b_{31}\partial^2\lambda+b_{32}\partial\lambda^2+b_{33}\lambda^3$.
Plugging this into \eqref{L3-3} gives $b_{30}=b_{31}=0$ and thus $f_3(\partial,\lambda)=b_{32}\partial\lambda^2+b_{33}\lambda^3$. By Lemma \ref{lem4} with $\phi(\partial)=\partial^2$, we may subtract from $f_3(\partial,\lambda)$ the trivial cocycle of the form  $(2b_{33}-b_{32})\big((\partial+\Delta\lambda)\phi(\partial)-(\partial+\bar\Delta\lambda)\phi(\partial+\lambda)\big)$, and thus assume that $f_3(\partial,\lambda)=b(2\partial+\lambda)\la^2$, for some $b\in\C.$

 Hence $g(\pa,\la)=a_1(\pa+\la)\la$ and $f(\partial,\lambda)=f_2(\pa,\la)+f_3(\pa,\la)=a_1(\la^2-\partial^2)+b(2\partial+\lambda)\la^2$ for $\Delta=1$ and $\bar\Delta=-1$.
The corresponding extension is nontrivial for $b$ arbitrary and
$a_1\neq0$.

\vskip5pt
{\bf Case 4.} $n=3$.
\vskip5pt

In this case, $\Delta-\bar\Delta=3$ and  $g(\pa,\la)=a_1\pa^2\la+a_2\pa\la^2+a_3\la^3$ for $(a_1,a_2,a_3)\neq(0,0,0)$. Substituting this into \eqref{H3} and then comparing the coefficients of $ \la^3$ and $\la^2$ respectively, we have
$$a_3=-\bar\Delta a_1, a_2=-2\bar\Delta a_1-a_1, 3a_3=-\bar\Delta a_2,$$
which yields $\bar\Delta=-2$ and $\Delta=3+\bar\Delta=1$. Therefore
\begin{eqnarray}\label{&}
g(\pa,\la)=a_1\la(\pa^2+3\pa\la+2\la^2), \ a_1\neq 0.
\end{eqnarray}
As before,
let $f_i(\partial,\lambda)$ be the $i$th homogenous bifurcation of $f(\partial,\lambda)$. By Lemmas \ref {lem4} and \ref{important}, we have $0\leq i \leq 4$ into $f_3(\partial,\lambda)\neq0$.
Write
$f_3(\partial,\lambda)=b_{30}\partial^3+b_{31}\partial^2\lambda+b_{32}\partial\lambda^2+b_{33}\lambda^3$. Plugging this and \eqref{&} into \eqref{L3} with $\la=0$ gives, after simplification,
\begin{eqnarray}
-a_1(\pa^2\mu^2+3\pa\mu^3+2\mu^4)=b_{30}(3\pa^2\mu^2+5\pa\mu^3+2\mu^4).
\end{eqnarray}
It follows $a_1=-3b_{30}$, $3a_1=-5b_{30}$ and $2a_1=-2b_{30}$, leading to $a_1=0$. A contradiction. Therefore there is no nontrivial solution in this case.

\vskip5pt
According to the above discussion and Theorem \ref{th4}, we have the following theorem.
\vskip-5pt
\begin{theo}\label{theo4}
Nontrivial extensions of conformal $\mathcal{R}$-modules of the form \eqref{3m} exist only if $\alpha=\bar\alpha$ and $\Delta-\bar\Delta=0,1,2,3,4,5,6$. The following is a complete list of values of $\Delta$ and $\bar\Delta$, along with the pair of polynomials $f(\partial,\lambda)$ and $g(\partial,\lambda)$, where $\bar\partial=\partial+\alpha$:
\begin{itemize}
\item[{\rm (i)}] $\Delta=\bar\Delta,$ $g(\pa,\la)=0,$ $f(\pa,\la)=c_0+c_1\la,$ $(c_0,c_1)\neq(0,0).$
\item [{\rm (ii)}]$\Delta-\bar\Delta=1$, $g(\pa,\la)=a_1\lambda,$ $f(\pa,\la)=\frac{a_1}{\bar\Delta}\bar\pa,$ $a_1\neq0$.
\item [{\rm (iii)}]
$\Delta-\bar\Delta=2,$ $g(\pa,\la)=0$, $f(\pa,\la)=b_1\la^2(2\bar\pa+\la)$, $b_1\neq0$.
\item [{\rm (iii')}] $\Delta=1$, $\bar\Delta=-1$, $g(\pa,\la)=a_1(\bar\pa+\la)\lambda,$ $f(\pa,\la)=a_1(\la^2-\bar\pa^2)+b(2\bar\pa+\la)\la^2,$ $a_1\neq0$, $b\in\C$.
\item [{\rm (iv)}]
$\Delta-\bar\Delta=3,$ $g(\pa,\la)=0$,  $f(\pa,\la)=b_1\bar\pa\la^2(\bar\pa+\la)$, $b_1\neq 0$.
\item [{\rm (v)}]$\Delta-\bar\Delta=4,$ $g(\pa,\la)=0$,  $f(\pa,\la)=b_1\la^2(4\bar\pa^3+6\bar\pa^2\la-\bar\pa\la^2+\bar\Delta\la^3)$, $b_1\neq 0$.
\item [{\rm (vi)}]$(\Delta,\bar\Delta)=(1,-4),$ $g(\pa,\la)=0$, $f(\pa,\la)=b_1\big(\bar\pa^4\la^2-10\bar\pa^2\la^4-17\bar\pa\la^5-8\la^6\big)$, $b_1\neq 0$.
\item [{\rm (vii)}]$(\Delta,\bar\Delta)=(\frac72\pm\frac{\sqrt{19}}2, -\frac52\pm\frac{\sqrt{19}}2),$ $g(\pa,\la)=0$,   $f(\pa,\la)=b_1\big(\bar\pa^4\la^3-(2\bar\Delta+3)\bar\pa^3\la^4-3\bar\Delta\bar\pa^2\la^5-(3\bar\Delta+1)\bar\pa\la^6-(\bar\Delta+\frac9{28})\la^7\big)$, $b_1\neq 0$.
\end{itemize}
\end{theo}

\noindent\bf{ Acknowledgements.}\ \rm
{\footnotesize This work was supported by National Natural Science
Foundation of China (11301109) and the Research Fund for the Doctoral Program of Higher Education (20132302120042).}

\end{document}